\documentclass[12pt]{article}
\usepackage{amsmath,amsthm,amssymb}

\allowdisplaybreaks[3]

\numberwithin{equation}{section}

\newtheorem{theorem}{Theorem}[section]

\newtheorem{lemma}[theorem]{Lemma}
\newtheorem{proposition}[theorem]{Proposition}
\theoremstyle{remark}
\newtheorem{remark}[theorem]{Remark}

\newcommand\R{{\mathbb R}}
\newcommand\X{{\R^d}}
\newcommand\N{{\mathbb N}}

\newcommand\B{{\mathcal B}}

\renewcommand\L{{\mathcal L}}
\newcommand\K{{\mathcal K}}
\newcommand\La{\Lambda}
\newcommand\la{\lambda}
\newcommand\Ga{\Gamma}
\newcommand\ga{\gamma}
\newcommand\eps{\varepsilon}

\newcommand\n{{|\eta|}}
\newcommand\ren{{\eps, \, \mathrm{ren}}}
\newcommand\lv{\left\vert}
\newcommand\rv{\right\vert}

\newcommand\lu{\left\langle}
\newcommand\ru{\right\rangle}
\newcommand\llu{\lu\!\lu}
\newcommand\rru{\ru\!\ru}

\newcommand\hT{{\hat{T}}}

\newcommand\hL{{\hat{L}}}

\newcommand\goto{\rightarrow}
\newcommand\hLrenadj{\hL_\ren^\ast}

\newcommand\e{{(\eps)}}

\newcommand\func[1]{{\mathrm{#1}}}

\DeclareMathOperator*\esssup{{\mathrm{ess\,sup}}}

\begin{document}

\title{Operator approach to Vlasov scaling for~some~models~of~spatial
ecology}

\author{Dmitri Finkelshtein%
\thanks{Institute of Mathematics, National Academy of Sciences of Ukraine, Kyiv,
Ukraine (\texttt{fdl@imath.kiev.ua}).}
\and Yuri Kondratiev%
\thanks{Fakult\"{a}t f\"{u}r Mathematik, Universit\"{a}t Bielefeld, 33615
Bielefeld, Germany (\texttt{kondrat@math.uni-bielefeld.de})} \and Oleksandr Kutoviy%
\thanks{%
Fakult\"{a}t f\"{u}r Mathematik, Universit\"{a}t Bielefeld, 33615
Bielefeld, Germany (\texttt{kutoviy@math.uni-bielefeld.de}).}}

\date{}

\maketitle

\pagestyle{myheadings} \thispagestyle{plain} \markboth{D.
Finkelshtein, Yu. Kondratiev and O. Kutoviy}{Operator approach to
Vlasov scaling for~models~of~spatial ecology}



\begin{abstract}
We consider Vlasov-type scaling for Markov evolution of
birth-and-death type in continuum, which is based on a proper
scaling of corresponding Markov generators and has an algorithmic
realization in terms of related hierarchical chains of correlation
functions equations. The existence of rescaled and limiting
evolutions of correlation functions as well as convergence to the
limiting evolution are shown. The obtained results enable to derive
a non-linear Vlasov-type equation for the density of the limiting
system.
\end{abstract}

\newpage

\section{Introduction}
The Vlasov equation is a famous example of  a kinetic equation which
describes the dynamical behavior of a many-body system. In physics,
it characterizes the Hamiltonian motion of an infinite particle
system influenced by weak long-range forces in the mean field
scaling limit. The detailed exposition of the Vlasov scaling for the
Hamiltonian dynamics was given by W.Braun and K.Hepp \cite{BH1977}
and later by R.L.Dobrushin \cite{Dob1979} for more general
deterministic dynamical systems.   The limiting Vlasov-type
equations for particle densities in both papers are considered in
classes of integrable functions  (or finite measures in the weak
form). This corresponds, actually, to the situation of finite volume
systems or systems with zero mean density in an infinite volume. The
Vlasov equation for the integrable functions was investigated in
details by V.V.Kozlov \cite{Koz2008}. An excellent   review about
kinetic equations which describe dynamical multi-body systems was
given by H.Spohn \cite{Spo1980}, \cite{Spo1991}. Note that in the
framework of interacting diffusions a similar problem is known as
the McKean--Vlasov limit.

Motivated by the study of Vlasov scaling for some classes of
stochastic evolutions in continuum for which the use of the
mentioned above approaches breaks down (even in the finite volumes)
we developed the general approach to study the Vlasov-type dynamics
(see \cite{FKK2010b}). It is based on a proper scaling of the
hierarchical equations for the evolution of correlation functions
and can be interpreted  in the terms of the rescaled Markov
generators.  Up to our knowledge, at the present time only  this
technique may give a possibility to control the convergence in the
Vlasov limit in the case of non-integrable densities which is
generic for infinite volume infinite particle systems.
 Saying about the evolutions, the kinetic equations of which can not be
 studied by the classical techniques described in \cite{BH1977} and \cite{Dob1979},  we have in
mind, first of all,  spatial birth-and-death Markov processes (e.g.,
continuous Glauber dynamics, spatial ecological models) and hopping
particles Markov evolutions (e.g., Kawasaki dynamics in continuum).
The main difficulty to carry out the approach proposed by W.Braun,
K.Hepp \cite{BH1977} and R.L.Dobrushin \cite{Dob1979} for such
models is absence of the proper descriptions in terms of stochastic
evolutional equations. Another problem concerns the possible
variation of particles number in the evolution. The important point
to note also is that an application of the technique proposed in
\cite{FKK2010b} leads to a  limiting hierarchy which posses a chaos
preservation property.

\newpage The aim of this paper is to study the Vlasov scaling for the
individual based model (IBM) in spatial ecology introduced by
B.Bolker and S.Pacala \cite{BP1997, BP1999}, U.Dieckmann and R.Law
\cite{DL2000} (BDLP model) using the scheme developed in
\cite{FKK2010b}. A population in this model is represented by a
configuration of motionless organisms (plants) located in an
infinite habitat (an Euclidean space in our considerations).  The
unbounded
 habitat is taken to avoid boundary
 effects in the population evolution.

 The evolution equation for the correlation functions of the BDLP
model was studied in details in \cite{FKK2009}. In \cite{BP1997,
BP1999}, \cite{DL2000}  this system was called the system of spatial
moment equations for plant competition and, actually, this system
itself was taking as a definition of the dynamics in the BDLP model.
The mathematical structure of the correlation functions evolution
equation is close to other well-known hierarchical systems in
mathematical physics, e.g., BBGKY hierarchy for the Hamiltonian
dynamics (see, e.g. \cite{DSS1989}). As in all hierarchical chains
of equations, we can not expect the explicit form of the solution,
and even more, the existence problem for these equations is a highly
delicate question.

According to the general scheme (see \cite{FKK2010b}), we state
conditions on structural coefficients of the BDLP Markov generator,
which give a weak convergence of  the rescaled generator to the
limiting generator of  the related Vlasov hierarchy. Next, we may
compute limiting Vlasov type equation for the BDLP model leaving
 the question about the strong convergence of the
hierarchy solutions for a separate analysis.
 A control of  the strong convergence of the rescaled hierarchy  is,
 in general,  a difficult technical problem.  In particular, this problem remains be
 open for BBGKY hierarchy for the case of Hamiltonian dynamics as well
  as for Bogoliubov--Streltsova hierarchy corresponding to the gradient diffusion model.
In the present paper we show the existence of  the rescaled and limiting
evolutions of correlation functions related to the Vlasov scaling of the BDLP model
 and the convergence to the limiting evolution. With this evolution for special class of initial conditions is related a non-linear equation for the density, which is called Vlasov equation for the considered stochastic dynamics.

Let us mention  that a version of the BDLP model for the case of
finite populations was studied in the  paper \cite{FM2004}. In this
work the authors developed  a probabilistic representation for the
finite BDLP process and applied this technique to analyze a
mean-field limit in the spirit of  classical Dobrushin or
McKean--Vlasov  schemes. They obtained an integro-differential
equation for the limiting deterministic process corresponding to an
integrable initial condition. The latter equation coincides with the
Vlasov equation for the BDLP model derived below in our approach.

The present paper is organized in the following way. Section 2 is devoted to the general settings required for the description of the model which we study.  In Subsection 3.1 we discuss the general Vlasov scaling approach for spatial continuos models. Subsection 3.2 is devoted to the abstract convergence result for semigroups in Banach spaces which will be crucial to prove the main statements of the paper presented in Subsection 3.3. The corresponding proofs are given in Subsection 3.4.

\section{Basic fact and description of model}\label{sect:base}
\subsection{General facts and notations}

Let $\B({\R}^{d})$ be the family of all Borel sets in $\R^d$ and $\B_{b}({\R}^{d})$ denotes the system of all bounded sets in
$\B({\R}^{d})$.

The space of $n$-point configuration is
\begin{equation*}
\Ga _{0}^{(n)}=\Ga _{0,{\R}^{d}}^{(n)}:=\left\{ \left. \eta \subset
{\R}^{d}\right| \,|\eta |=n\right\} ,\quad n\in \N_0:=\N\cup \{0\},
\end{equation*}
where $|A|$ denotes the cardinality of the set $A$. The space
$\Ga_{\La}^{(n)}:=\Ga _{0,\La }^{(n)}$ for $\La \in \B_b({\R}^{d})$
is defined analogously to the space $\Ga_{0}^{(n)}$. As a set, $\Ga
_{0}^{(n)}$ is equivalent to the symmetrization of
\begin{equation*}
\widetilde{({\R}^{d})^n} = \left\{ \left. (x_1,\ldots ,x_n)\in
({\R}^{d})^n\right| \,x_k\neq x_l\,\,\mathrm{if} \,\,k\neq l\right\}
,
\end{equation*}
i.e. $\widetilde{({\R}^{d})^n}/S_{n}$, where $S_{n}$ is the
permutation group over $\{1,\ldots,n\}$. Hence one can introduce the
corresponding topology and Borel $\sigma $-algebra, which we denote
by $O(\Ga_{0}^{(n)})$ and $\B(\Ga_{0}^{(n)})$, respectively. Also
one can define a measure $m^{(n)}$ as an image of the product of
Lebesgue measures $dm(x)=dx$ on $\bigl(\R^d, \B(\R^d)\bigr)$.

The space of finite configurations
\begin{equation*}
\Ga _{0}:=\bigsqcup_{n\in \N_0}\Ga _{0}^{(n)}
\end{equation*}
is equipped with the topology which has structure of disjoint union.
Therefore, one can define the corresponding Borel $\sigma $-algebra
$\B (\Ga _0)$.

A set $B\in \B (\Ga _0)$ is called bounded if there exists $\La \in
\B_b({\R}^{d})$ and $N\in \N$ such that $B\subset
\bigsqcup_{n=0}^N\Ga _\La ^{(n)}$. The Lebesgue--Poisson measure
$\la_{z} $ on $\Ga_0$ is defined as
\begin{equation*}
\la _{z} :=\sum_{n=0}^\infty \frac {z^{n}}{n!}m ^{(n)}.
\end{equation*}
Here $z>0$ is the so called activity parameter. The restriction of
$\la _{z} $ to $\Ga _\La $ will be also denoted by $\la _{z} $.

The configuration space
\begin{equation*}
\Ga :=\left\{ \left. \ga \subset {\R}^{d}\ \right| \; |\ga \cap \La
|<\infty, \text{ for all } \La \in \B_b({\R}^{d})\right\}
\end{equation*}
is equipped with the vague topology. It is a Polish space (see e.g.
\cite{KK2006}). The corresponding  Borel $\sigma $-algebra $ \B(\Ga
)$ is defined as the smallest $\sigma $-algebra for which all
mappings $N_\La :\Ga \rightarrow \N_0$, $N_\La (\ga ):=|\ga \cap \La
|$ are measurable, i.e.,
\begin{equation*}
\B(\Ga )=\sigma \left(N_\La \left| \La \in
\B_b({\R}^{d})\right.\right ).
\end{equation*}
One can also show that $\Ga $ is the projective limit of the spaces
$\{\Ga _\La \}_{\La \in \B_b({\R}^{d})}$ w.r.t. the projections
$p_\La :\Ga \rightarrow \Ga _\La $, $p_\La (\ga ):=\ga _\La $, $\La
\in \B_b({\R}^{d})$.

The Poisson measure $\pi _{z} $ on $(\Ga ,\B(\Ga ))$ is given as the
projective limit of the family of measures $\{\pi _{z} ^\La \}_{\La
\in \B_b({\R}^{d})}$, where $\pi _{z} ^\La $ is the measure on $\Ga
_\La $ defined by $\pi _{z} ^\La :=e^{-z m (\La )}\la _{z}$.

We will use the following classes of functions:
$L_{\mathrm{ls}}^0(\Ga _0)$ is the set of all measurable functions
on $\Ga_0$ which have a local support, i.e. $G\in
L_{\mathrm{ls}}^0(\Ga _0)$ if there exists $\La \in \B_b({\R}^{d})$
such that $G\upharpoonright_{\Ga _0\setminus \Ga _\La }=0$;
$B_{\mathrm{bs}}(\Ga _0)$ is the set of bounded measurable functions
with bounded support, i.e. $G\upharpoonright_{\Ga _0\setminus B}=0$
for some bounded $B\in \B(\Ga_0)$.

On $\Ga $ we consider the set of cylinder functions
$\mathcal{F}_{\mathrm{cyl}}(\Ga )$, i.e. the set of all measurable functions $G$
on $\bigl(\Ga,\B(\Ga))\bigr)$ which are measurable w.r.t. $\B_\La
(\Ga )$ for some $\La \in \B_b({\R}^{d})$. These functions are
characterized by the following relation: $F(\ga )=F\upharpoonright
_{\Ga _\La }(\ga _\La )$.

The following mapping between functions on $\Ga _0$, e.g.
$L_{\mathrm{ls}}^0(\Ga _0)$, and functions on $\Ga $, e.g.
$\mathcal{F}_{\mathrm{cyl}}(\Ga )$, plays the key role in our further
considerations:
\begin{equation}
KG(\ga ):=\sum_{\eta \Subset \ga }G(\eta ), \quad \ga \in \Ga,
\label{KT3.15}
\end{equation}
where $G\in L_{\mathrm{ls}}^0(\Ga _0)$, see e.g.
\cite{KK2002,Len1975,Len1975a}. The summation in the latter
expression is taken over all finite subconfigurations of $\ga ,$
which is denoted by the symbol $\eta \Subset \ga $. The mapping $K$
is linear, positivity preserving, and invertible, with
\begin{equation}
K^{-1}F(\eta ):=\sum_{\xi \subset \eta }(-1)^{|\eta \setminus \xi
|}F(\xi ),\quad \eta \in \Ga _0.\label{k-1trans}
\end{equation}

Let $ \mathcal{M}_{\mathrm{fm}}^1(\Ga )$ be the set of all
probability measures $\mu $ on $\bigl( \Ga, \B(\Ga) \bigr)$ which
have finite local moments of all orders, i.e. $\int_\Ga |\ga _\La
|^n\mu (d\ga )<+\infty $ for all $\La \in \B_b(\R^{d})$ and $n\in
\N_0$. A measure $\rho $ on $\bigl( \Ga_0, \B(\Ga_0) \bigr)$ is
called locally finite iff $\rho (A)<\infty $ for all bounded sets
$A$ from $\B(\Ga _0)$. The set of such measures is denoted by
$\mathcal{M}_{\mathrm{lf}}(\Ga _0)$.

One can define a transform $K^{*}:\mathcal{M}_{\mathrm{fm}}^1(\Ga
)\rightarrow \mathcal{M}_{\mathrm{lf}}(\Ga _0),$ which is dual to
the $K$-transform, i.e., for every $\mu \in
\mathcal{M}_{\mathrm{fm}}^1(\Ga )$, $G\in \B_{\mathrm{bs}}(\Ga _0)$
we have
\begin{equation*}
\int_\Ga KG(\ga )\mu (d\ga )=\int_{\Ga _0}G(\eta )\,(K^{*}\mu
)(d\eta ).
\end{equation*}
The measure $\rho _\mu :=K^{*}\mu $ is called the correlation
measure of $\mu $.

As shown in \cite{KK2002} for $\mu \in
\mathcal{M}_{\mathrm{fm}}^1(\Ga )$ and any $G\in L^1(\Ga _0,\rho
_\mu )$ the series \eqref{KT3.15} is $\mu$-a.s. absolutely
convergent. Furthermore, $KG\in L^1(\Ga ,\mu )$ and
\begin{equation}
\int_{\Ga _0}G(\eta )\,\rho _\mu (d\eta )=\int_\Ga (KG)(\ga )\,\mu
(d\ga ). \label{Ktransform}
\end{equation}

A measure $\mu \in \mathcal{M}_{\mathrm{fm} }^1(\Ga )$ is called
locally absolutely continuous w.r.t. $\pi _{z} $ iff $\mu_\La :=\mu
\circ p_\La ^{-1}$ is absolutely continuous with respect to $\pi
_{z} ^\La $ for all $\La \in \B_\La ({\R}^{d})$. In this case $\rho
_\mu :=K^{*}\mu $ is absolutely continuous w.r.t $\la _{z} $. We
denote
\begin{equation*}
k_{\mu}(\eta):=\frac{d\rho_{\mu}}{d\la_{z}}(\eta),\quad
\eta\in\Ga_{0}.
\end{equation*}
The functions
\begin{equation}
k_{\mu}^{(n)}:(\R^{d})^{n}\longrightarrow\R_{+}\label{corfunc}
\end{equation}
\begin{equation*} k_{\mu}^{(n)}(x_{1},\ldots,x_{n}):=\left\{\begin{array}{ll}
k_{\mu}(\{x_{1},\ldots,x_{n}\}), & \mbox{if $(x_{1},\ldots,x_{n})\in
\widetilde{(\R^{d})^{n}}$}\\ 0, & \mbox{otherwise}
\end{array}
\right.\end{equation*} are the correlation functions well known in
statistical physics, see e.g. \cite{Rue1969}, \cite{Rue1970}.


\subsection{Description of model}

We consider the evolving in time system of interacting individuals (particles) in the
space $\R^{d}$. The state of the system at the
fixed moment of time $t>0$ is described by the random configuration
$\ga_{t}$ from $\Ga$. Heuristically, the mechanism of the evolution is given by a
Markov generator, which has the following form
\[
L:=L^- + L^+ ,
\]
where
\begin{align}
(L^-F)(\ga)&:=(L^-(m,\varkappa^-,a^-)F)(\ga):=
\sum_{x\in\ga}\left[m+\varkappa^{-}\sum_{y\in\ga\setminus
x}a^{-}(x-y)\right]D_{x}^{-}F(\ga),\notag\\
(L^+F)(\ga)&:=(L^+(\varkappa^+,a^-)F)(\ga):=
\varkappa^{+}\int_{\R^{d}}\sum_{y\in\ga}a^{+}(x-y)D_{x}^{+}F(\ga)dx.
\label{BP-gen}
\end{align}
Here $0\leq a^{-}, \,a^{+}\in L^{1}(\R^{d})$ are arbitrary, even functions
such that
\begin{equation*}
\int_{\R^{d}}a^{-}(x)dx =\int_{\R^{d}}a^{+}(x)dx =1
\end{equation*}
(in other words, $a^{-},\,a^{+}$ are probability densities) and $m$,
$\varkappa^{-}$, $\varkappa^{+}>0$ are some positive constants.

The pre-generator $L$ describes the Bolker--Dieckmann--Law--Pacala BDLP
mo\-del, which was introduced in \cite{BP1997,BP1999,DL2000}. During
the corresponding stochastic evolution the birth of individuals
occurs independently and the death is ruled not only by the global
regulation (mortality) but also by the local regulation with the
kernel $\varkappa^-a^-$. This regulation may be described as a
competition (e.g., for resources) between individuals in the
population.

The evolution of the one dimensional distribution for such systems
can be expressed in terms of their characteristics, e.g. the
correlation functions (see \eqref{corfunc}). The dynamics of correlation functions for the
BDLP model was studied in \cite{FKK2009}. The main result of this
paper informally says the following:

{\em If the mortality $m$ and the competition kernel
$\varkappa^-a^-$ are large enough, then the dynamics of correlation
functions, associated with the pre-generator \eqref{BP-gen}, exists and preserves
(sub-)Poissonian bound.}

For the readers convenience we repeat below the relevant material
from \cite{FKK2009} without proofs.

Let $\hat{L}^\pm:= K^{-1} L^\pm K$ be the $K$-image of $L^\pm$,
which can be initially defined on functions from
$B_{\mathrm{bs}}(\Ga _0)$. For arbitrary and fixed $C>0$ we consider
the operator $\hat{L}:=\hat{L}^+ + \hat{L}^-$ in the functional
space
\begin{equation*}
\mathcal{L}_{C}=L^{1}\left( \Gamma _{0},C^{\left\vert \eta
\right\vert }d\lambda \left( \eta \right) \right) .
\end{equation*}
Below, symbol $\left\Vert \cdot \right\Vert_{C} $ stands for the
norm of this space.

For any $\omega>0$ we define $\mathcal{H}(\omega)$ to be the set of
all densely defined closed operators $T$ on $\mathcal{L}_{C},$ the
resolvent set $\rho(T)$ of which contains the sector
\begin{equation*}
\func{Sect}\left(\frac{\pi}{2}+\omega\right):=\left\{\zeta\in\mathbb{C}\,\Bigm|
|\func{arg}\, \zeta|<\frac{\pi}{2}+\omega\right\},
\end{equation*}
and for any $\varepsilon >0$
\begin{equation*}
||(T-\zeta 1\!\!1)^{-1}||\leq \frac{M_{\varepsilon}}{|\zeta|},\quad
|\arg\,\zeta |\leq\frac{\pi}{2}+\omega-\varepsilon,
\end{equation*}
where $M_{\varepsilon}$ does not depend on $\zeta$.

The first non-trivial result, which is based on the perturbation theory,
says that the operator $\hat{L}$ with the domain
\begin{equation*}
D(\hat{L}):=\left\{ G\in \mathcal{L}_{C} \Bigm| \left\vert
\cdot \right\vert G(\cdot) \in \mathcal{L}_{C}, \;E^{a^{-}}(\cdot) G(\cdot) \in \mathcal{L}_{C}\right\}
\end{equation*}
is a generator of a holomorphic $C_{0}$-semigroup $\hat{U}_t$
on $\L_C$.

To construct the corresponding evolution of correlation functions we note that the dual space
$(\L_C)'=\bigl(L^1(\Ga_0, d\la_C)\bigr)'=L^\infty(\Ga_0, d\la_C)$, where $d\la_C:= C^{|\cdot|} d\la$.
The space $(\L_C)'$ is isometrically isomorphic to the Banach
space
\begin{equation*}
\K_{C}:=\left\{k:\Ga_{0}\goto\R\,\Bigm| k(\cdot) C^{-|\cdot|}\in
L^{\infty}(\Ga_{0},\la)\right\}
\end{equation*}
with the norm
\begin{equation*}
\|k\|_{\K_C}:=\|C^{-|\cdot|}k(\cdot)\|_{L^{\infty}(\Ga_{0},\la)},
\end{equation*}
where the isomorphism is provided by the isometry $R_C$
\begin{equation}\label{isometry}
(\L_C)'\ni k  \longmapsto R_Ck:=k(\cdot) C^{|\cdot|}\in \K_C.
\end{equation}
In fact, we have duality between Banach spaces $\L_C$
and $\K_C$ given by the following expression
\begin{equation}
\llu  G,\,k \rru  := \int_{\Ga_{0}}G\cdot k\, d\la,\quad G\in\L_C, \
k\in \K_C \label{duality}
\end{equation}
with
\begin{equation}
\lv  \llu   G,k \rru  \rv \leq \|G\|_C \cdot\|k\|_{\K_C}.
\label{funct_est}
\end{equation}
It is clear that for any $k\in \K_C$
\begin{equation}\label{RB-norm}
|k(\eta)|\leq \|k\|_{\K_C} \, C^{|\eta|} \quad \text{for } \la\text{-a.a. }
\eta\in\Ga_0.
\end{equation}

Let $\hL '$ be the adjoint operator to $\hL $ in $(\L_C)'$ with domain
$D(\hL ')$. Its image in $\K_C$ under the isometry
$R_C$ we denote by $\hL^{*}=R_C\hL 'R_{C^{-1}}$. It is evident that the domain of $\hL^{*}$ will be $D(\hL
^{*})=R_C
 D(\hL ')$, correspondingly. Then, for any $G\in\L_C$, $k\in D(\hL^\ast)$
\begin{align*}
\int_{\Ga_0}G\cdot \hL^\ast k d\la&=\int_{\Ga_0}G\cdot R_C\hL
'R_{C^{-1}} k d\la=\int_{\Ga_0}G\cdot \hL 'R_{C^{-1}} k d\la_C\\&=
\int_{\Ga_0}\hL G\cdot R_{C^{-1}} k d\la_C=\int_{\Ga_0}\hL G\cdot k
d\la,
\end{align*}
therefore, $\hL^\ast$ is the dual operator to $\hL $ w.r.t. the duality
\eqref{duality}. By \cite{FKO2009}, we have the precise form of
$\hL^{*}$:
\begin{align}\label{dual-descent}
(\hL^* k)(\eta)=&-\left(m|\eta|+\varkappa^{-}E^{a^{-}}(\eta)\right)k(\eta)\\&+
\varkappa^{+}\sum_{x\in\eta}\sum_{y\in\eta\setminus x}a^{+}(x-y)k(\eta\setminus x)\notag\\&
+\varkappa^{+}\int_{\R^{d}}\sum_{y\in\eta}a^{+}(x-y)k((\eta\setminus y)\cup
x)dx\notag\\&
+\varkappa^{-}\int_{\R^{d}}\sum_{y\in\eta}a^{-}(x-y)k(\eta\cup x)dx.
\notag
\end{align}
Now we consider the adjoint semigroup $\hT '(t)$ on $(\L_C)'$ and
its image $\hT^\ast(t)$ in $\K_C$. The latter one describes the
evolution of correlation functions. Transferring the general results
about adjoint semigroups (see, e.g., \cite{EN2000}) onto semigroup
$\hT^\ast(t)$ we deduce that it will be weak*-continuous and
weak*-differentiable at $0$. Moreover, $\hL^\ast$ will be the
weak*-generator of $\hT^\ast(t)$. Here and subsequently we mean
``weak*-properties'' w.r.t. the duality \eqref{duality}.
\section{Vlasov scaling}
\subsection{Description of Vlasov scaling}
We begin with a general idea of the Vlasov-type
scaling. It is of interest to construct some scaling $L_\eps$, $\eps>0$ of the generator $L$,
such that the following scheme works.

Suppose that we know the proper scaling of $L$ and we are able to prove the existence of
 the semigroup $\hT_\eps(t)$ with the generator $\hL
_\eps:=K^{-1} L_\eps K$ in the space $\L_C$ for some $C>0$.
Let us consider the
Cauchy problem corresponding to the adjoint operator $\hL^\ast$ and take an initial
function with the strong singularity in $\eps$. Namely,
$$k_0^\e(\eta) \sim \eps^{-|\eta|} r_0(\eta),\quad\quad\eps\goto 0,\quad\quad
\eta\in\Ga_0,$$ where the function $r_0$ is independent of $\eps$.
The solution to this problem is described by the dual semigroup
$\hT_\eps^\ast(t)$. The scaling $L\mapsto L_\eps$ has to be chosen
in such a way that $\hT_\eps^\ast(t)$  preserves the order of the
singularity:
$$(\hT_\eps^\ast(t)k_0^\e)(\eta) \sim \eps^{-|\eta|} r_t(\eta),\quad\quad\eps\goto 0,
\quad\quad\eta\in\Ga_0.$$ Another very important requirement for the proper scaling concerns the
dynamics $r_0 \mapsto r_t$. It should preserve Lebesgue--Poisson
exponents: if $r_0(\eta)=e_\la(\rho_0,\eta)$ then
$r_t(\eta)=e_\la(\rho_t,\eta)$ and there exists explicit (nonlinear,
in general) differential equation for $\rho_t$
\begin{equation}\label{V-eqn-gen}
\frac{\partial}{\partial t}\rho_t(x) = \upsilon(\rho_t(x)),
\end{equation}
which will be called the Vlasov-type equation.

Now let us explain the main technical steps to realize Vlasov-type scaling. Let us
consider for any $\eps>0$ the following mapping (cf. \eqref{isometry})
on functions on $\Ga_0$
\begin{equation}
(R_\eps r)(\eta):=\eps^\n r(\eta).
\end{equation}
This mapping is ``self-dual'' w.r.t. the duality \eqref{duality},
moreover, $R_\eps^{-1}=R_{\eps^{-1}}$. Then we have $k^\e_0\sim
R_{\eps^{-1}} r_0$, and we need $r_t \sim R_\eps
\hT_\eps^\ast(t)k_0^\e \sim  R_\eps \hT _\eps^\ast(t)R_{\eps^{-1}}
r_0$. Therefore, we have to show that for any $t\geq 0$ the operator
family $R_\eps \hT _\eps^\ast(t)R_{\eps^{-1}}$, $\eps>0$ has
limiting (in a proper sense) operator $U(t)$ and
\begin{equation}\label{chaospreserving}
U(t)e_\la(\rho_0)=e_\la(\rho_t).
\end{equation}
But, informally, $\hT^\ast_\eps(t)=\exp{\{t\hL^\ast_\eps\}}$ and
$R_\eps \hT_\eps^\ast(t)R_{\eps^{-1}}=\exp{\{t R_\eps \hL_\eps^\ast
R_{\eps^{-1}} \}}$. Let us consider the ``renormalized'' operator
\begin{equation}\label{renorm_def}
\hLrenadj :=  R_\eps \hL_\eps^\ast R_{\eps^{-1}}.
\end{equation}
In fact, we need that there exists an operator $\hat{V}^\ast$ (called Vlasov operator) such
that $\exp{\{t R_\eps \hL_\eps^\ast R_{\eps^{-1}} \}}\goto
\exp{\{t\hat{V}^\ast\}=:U(t)}$ for which \eqref{chaospreserving}
holds.
Hence, heuristic way to produce
the scaling $L\mapsto L_\eps$ is to demand that
\begin{equation*}
\lim_{\eps\goto 0}\left(\frac{\partial}{\partial
t}e_\la(\rho_t,\eta)-\hLrenadj e_\la(\rho_t,\eta)\right)=0, \quad
\eta\in\Ga_0,
\end{equation*}
if $\rho_t$ satisfies \eqref{V-eqn-gen}. The point-wise limit of
$\hLrenadj$ will be natural candidate for $\hat{V}^\ast$.
Having chosen the proper scaling we proceed to the following technical steps which give the
rigorous meaning to the idea introduced above. Note that definition  \eqref{renorm_def} implies $\hL_\ren=R_{\eps^{-1}}\hL
_\eps R_\eps$. We prove that  ``renormalized'' operator $\hL_\ren$ is a generator of a contraction semigroup $\hT_\ren(t)$ on
$\L_C$. Next we show that this semigroup converges strongly to some
semigroup $\hT_V(t)$ with the generator $\hat{V} $. This limiting semigroup leads us directly to the solution for the Vlasov-type equation. Below we show how to realize this scheme in details.

\subsection{Approximation in Banach space}

In this subsection we study general question about the strong convergence of semigroups in Banach spaces. The obtained results will be crucial in the realization of the Vlasov-type scaling for the BDLP model.

Let $\left\{U_{t}^{\varepsilon},\,t\geq 0\right\},\: \varepsilon\geq
0$ be a family of semigroups on a Banach space $E$. We set
$(L_{\varepsilon},\,D(L_{\varepsilon}))$ to be the generator of
$\left\{U_{t}^{\varepsilon},\,t\geq 0\right\}$ for each
$\varepsilon\geq 0$. Our purpose now is to describe the strong
convergence of semigroups $\left\{U_{t}^{\varepsilon},\,t\geq
0\right\},\: \varepsilon\geq 0$ in terms of the corresponding
generators as $\varepsilon$ tends to $0$.  According to the
classical result (see e.g. \cite{Kat1976}), it is enough to show
that there exists $\beta >0$ and $\la:\; \mathrm{Re}\,\la >\beta$ such that
\begin{equation}
\left(L_{\varepsilon}-\la1\!\!1\right)^{-1}\overset{s}{\longrightarrow}\left(L_{0}-\la1\!\!1\right)^{-1},\quad\quad\varepsilon\rightarrow 0,\label{resconv}
\end{equation}
where $1\!\!1$ is the identical operator.
In this subsection we show how to prove \eqref{resconv} under the
following assumptions on the family
$(L_{\varepsilon},\,D(L_{\varepsilon}))$, $\varepsilon \geq 0$:

\newpage{\bf { Assumptions (A)}}:
\begin{enumerate}
\item For any $\varepsilon\geq 0$, the operator $(L_{\varepsilon},\,D(L_{\varepsilon}))$ admits representation $$L_{\varepsilon}=A_{1}(\varepsilon)+A_{2}(\varepsilon),$$
where $A_{1}(\varepsilon)$ is a closed operator and $D(A_{1}(\varepsilon))=D(A_{2}(\varepsilon)):=D(L_{\varepsilon})$.
\item There exists $\beta>0$ and $\la$: $\mathrm{Re} \,\la >\beta$ such that
\begin{enumerate}
\item $\la$ belongs to the resolvent set of $A_{1}(\varepsilon)$ for any $\varepsilon \geq 0$ and
\begin{equation*}
\left( A_{1}(\varepsilon)-\lambda 1\!\!1\right) ^{-1}\overset{s}{
\longrightarrow }\left( A_{1}(0)-\lambda1\!\!1 \right) ^{-1},\varepsilon \rightarrow
0,
\end{equation*}
\item
\begin{equation*}
\sup_{\varepsilon >0}\left\Vert \left( A_{1}(\varepsilon)-\lambda 1\!\!1 \right)
^{-1}\right\Vert \leq \left\Vert \left( A_{1}(0)-\lambda 1\!\!1\right) ^{-1}\right\Vert ,
\end{equation*}
\item for any $\varepsilon \geq 0$
\begin{equation*}
\left\Vert A_{2}(\varepsilon)\left( A_{1}(\varepsilon)-\lambda 1\!\!1 \right)
^{-1}\right\Vert < 1,
\end{equation*}
\item
$\left( A_{2}(\varepsilon)\left( A_{1}(\varepsilon)
-\lambda 1\!\!1 \right) ^{-1}+1\!\!1 \right) ^{-1}$ converges strongly to the operator $\left( A_{2}(0)\left( A_{1}(0)-\lambda 1\!\!1 \right) ^{-1}+1\!\!1\right) ^{-1}$ as $\varepsilon \rightarrow 0$.
\end{enumerate}
\end{enumerate}

The strong convergence result for the family $\left\{U_{t}^{\varepsilon},\,t\geq 0\right\},\: \varepsilon\geq 0$ is established by our next theorem

\begin{theorem}\label{gentheor}
Let $(L_{\varepsilon},\,D(L_{\varepsilon}))$, $\varepsilon \geq 0$
be the family of generators corresponding to the family of
$C_{0}$-semigroups $\left\{U_{t}^{\varepsilon},\,t\geq 0\right\},\:
\varepsilon\geq 0$. Then, $U_{t}^{\varepsilon}$ converges strongly
to $U_{t}^{0}$ as $\varepsilon \rightarrow 0$ uniformly on each
finite interval of time, provided assumptions {\bf (A)} are
satisfied.
\end{theorem}
\begin{proof}
The proof is completed by showing \eqref{resconv}. For any
$\varepsilon\geq 0$ and $\lambda$ from the resolvent set of $A_{1}(\varepsilon)$ we have $$\mathrm{Ran}\left( \left(
A_{1}(\varepsilon)-\lambda 1\!\!1\right) ^{-1}\right) =D\left(
A_{1}(\varepsilon)\right) =D\left( A_{2}(\varepsilon)\right) =
D(L_{\varepsilon}). $$ Hence,
\begin{align}
& L_{\varepsilon}-\lambda1\!\!1=A_{1}(\varepsilon)+A_{2}(\varepsilon)-\lambda1\!\!1\nonumber \\
=&\left( A_{2}(\varepsilon)\left( A_{1}(\varepsilon)-\lambda 1\!\!1 \right) ^{-1}+1\!\!1\right)
\left( A_{1}(\varepsilon)-\lambda 1\!\!1\right).  \label{o1}
\end{align}
Combining \eqref{o1} with the assumption 2(c) of {\bf (A)} we get
the following representations for the resolvent
\begin{align}
&\left( L_{\varepsilon}-\lambda1\!\!1\right)^{-1}=\left( A_{1}(\varepsilon)+A_{2}(\varepsilon)-\lambda 1\!\!1
\right)
^{-1}  \label{longexpansresolv} \nonumber\\
=&\left( A_{1}(\varepsilon)-\lambda1\!\!1 \right) ^{-1}\left(
A_{2}(\varepsilon)\left( A_{1}(\varepsilon)-\lambda 1\!\!1\right)
^{-1}+1\!\!1\right) ^{-1}.
\end{align}
From this formula, triangle inequality and assumptions 2(a), 2(b) and 2(d) of {\bf (A)} we conclude the assertion of the theorem.
\end{proof}

\subsection{Main results}
We check at once that the proper scaling for the BDLP pre-generator is the following one
\begin{align}
(L_\eps F)(\ga):=&
\sum_{x\in\ga}\left[m+\eps\varkappa^{-}\sum_{y\in\ga\setminus
x}a^{-}(x-y)\right]D_{x}^{-}F(\ga)\label{resc BP-gen}
\\
&+\varkappa^{+}\int_{\R^{d}}\sum_{y\in\ga}a^{+}(x-y)D_{x}^{+}F(\ga)dx,
\quad\quad\quad \eps>0. \notag
\end{align}

Next we consider the formal $K$-image of $L_{\varepsilon}$ and the corresponding renormalized operator on  $B_{\mathrm{bs}}(\Ga_0)$:
\begin{equation*}
\hL_\eps G:=K^{-1}L_\eps K G; \qquad \hL_\ren G:=R_{\eps^{-1}}\hL
_\eps R_\eps G.
\end{equation*}
In the proposition below we calculate the precise form of the operator  $\hL_\ren$ for the BDLP model.
\begin{proposition}
For any $\varepsilon>0$ and any $G\in B_{\mathrm{bs}}\left( \Gamma _{0}\right) $
\begin{align*}
\hat{L}_{\varepsilon ,\mathrm{ren}}G =&A_{1}G+A_{2}G+\varepsilon
\left( B_{1}G+B_{2}G\right) ,
\end{align*}%
where%
\begin{align*}
(A_{1}G)\left( \eta \right)  =&-m\left\vert \eta \right\vert G\left(
\eta
\right) , \\
(A_{2}G)\left( \eta \right)  =&-\varkappa ^{-}\sum_{x\in \eta
}\sum_{y\in
\eta \setminus x}a^{-}\left( x-y\right) G\left( \eta \setminus x\right)  \\
&+\varkappa ^{+}\sum_{y\in \eta }\int_{\mathbb{R}^{d}}a^{+}\left(
x-y\right) G\left( \eta \setminus y\cup x\right) dx, \\
(B_{1}G)\left( \eta \right)  =&-\varkappa ^{-}E^{a^{-}}\left( \eta
\right)
G\left( \eta \right) , \\
(B_{2}G)\left( \eta \right)  =&\,\varkappa ^{+}\sum_{y\in \eta }\int_{\mathbb{%
R}^{d}}a^{+}\left( x-y\right) G\left( \eta \cup x\right) dx.
\end{align*}
\end{proposition}

\begin{proof}
According to the definition, we have $\hat{L}_{\varepsilon ,\mathrm{ren}}=R_{\varepsilon ^{-1}}\hat{L}%
_{\varepsilon }R_{\varepsilon }$, where $$\hat{L}_{\varepsilon }=\hat{L}%
^{-}\left( m,\varepsilon \varkappa ^{-}a^{-}\right) +\varepsilon ^{-1}\hat{L}%
^{+}\left( \varepsilon \varkappa ^{+}a^{+}\right).$$ As a result,%
\begin{align*}
(\hat{L}_{\varepsilon }G)\left( \eta \right) =&\,(A_{1}G)\left( \eta
\right)
+\varepsilon (B_{1}G)\left( \eta \right) +(B_{2}G)\left( \eta \right) \\
&-\varepsilon \varkappa ^{-}\sum_{x\in \eta }\sum_{y\in \eta
\setminus
x}a^{-}\left( x-y\right) G\left( \eta \setminus x\right) \\
&+\varkappa ^{+}\sum_{y\in \eta }\int_{\mathbb{R}^{d}}a^{+}\left(
x-y\right) G\left( \eta \setminus y\cup x\right) dx.
\end{align*}%
and hence
\begin{equation*}
(\hat{L}_{\varepsilon ,\mathrm{ren}}G)\left( \eta \right) =(A_{1}G)\left(
\eta \right) +(A_{2}G)\left(
\eta \right)+\varepsilon (\left( B_{1}+B_{2}\right) G)\left( \eta
\right) ,
\end{equation*}
which completes the proof.
\end{proof}
\begin{remark}
It is easily seen that the operator $\hat{V}: =A_{1}+A_{2}$ will be the point-wise limit of $\hL_\ren$ as $\varepsilon$ tends to 0. Therefore, the adjoint operator to $\hat{V}$ w.r.t. to the duality \eqref{duality} (if it exists) can be considered as a candidate for the Vlasov operator in our model.
\end{remark}

Below we give the rigorous meaning to the operator $\hat{L}_{\varepsilon ,\mathrm{ren}}$.
Let us  define the set
$$
D_{1}:=\left\{ G\in \mathcal{L}_{C}|\:
E^{a^{-}}\left(
\cdot  \right) G\left(
\cdot  \right)\in \mathcal{L}_{C},\;\left\vert \cdot \right\vert G\left(
\cdot  \right)\in \mathcal{L}_{C}\right\}
$$

\begin{proposition} \label{pr1}
For any $\varepsilon,\,m,\,\varkappa ^{-}, C>0$ the operator
\begin{equation}
A_{1}(\varepsilon):=A_{1}+\varepsilon B_{1}
\end{equation}
with the domain $D_{1}$
is a generator of a contraction $C_{0}$-semigroup on $\mathcal{L}_{C}$. Moreover, $A_{1}(\varepsilon)\in \mathcal{H}\left( \omega\right) $ for all $\omega
\in \left( 0;\frac{\pi }{2}\right) $.
\end{proposition}

\begin{proof}
See the proof of Proposition 4.2 in \cite{FKK2009}.
\end{proof}

\begin{remark}\label{rem001}
It is a simple matter to check that Proposition~\ref{pr1} holds also
in the case $\varepsilon=0$, provided the domain of the operator
$A_{1}(0):=A_{1}$ is changed to
$$
D_{0}:=\left\{ G\in \mathcal{L}_{C}\left|\right.\:
\;\left\vert \cdot \right\vert G\in \mathcal{L}_{C}\right\}\supset D_{1}.
$$
\end{remark}


The next task is to show that for any $\varepsilon > 0$ the operator
\begin{equation}
A_{2}(\varepsilon):= \hat{L}_{\varepsilon ,\mathrm{ren}}-A_{1}(\varepsilon)=A_{2}+\varepsilon B_{2}
\end{equation}
with the domain $D_{1}$ as well as  the operator $A_{2}(0):=A_{2}$ with the
domain $D_{0}$ are relatively bounded w.r.t. the operator
$(A_{1}(\varepsilon),\,D_{1})$ and $(A_{1},\,D_{0})$,
correspondingly. This is demonstrated in  Propositions~\ref{pr001}
and \ref{pr002}, which can be proved similarly to Lemmas 4.4 and 4.5
in \cite{FKK2009}.

\begin{proposition}\label{pr001}
For any $\delta >0$ and any $\,\varkappa^{-}, \varkappa^{+}, m, C >0$ such that
$$
\frac{\varkappa^{-}C}{m}+\frac{\varkappa^{+}}{m}\leq \delta
$$
the following estimate holds
\begin{equation*}
\left\Vert A_{2}G\right\Vert _{C}\leq \delta \left\Vert A_{1}G\right\Vert
_{C},~G\in D_{0}.
\end{equation*}%
Moreover, for all $\varepsilon >0$
\begin{equation*}
\left\Vert A_{2}G\right\Vert _{C}\leq \delta \left\Vert A_{1}(\varepsilon) G\right\Vert _{C},~G\in D_{1} .
\end{equation*}
\end{proposition}

\noindent Now, the operator $\left( A_{2}, D_{0}\right)$ is well-defined
on $\mathcal{L}_{C}$.

\begin{proposition} \label{pr002}
For any $\varepsilon ,\delta >0$ and any $ \varkappa ^{-},\varkappa ^{+}, m, C>0$
such that%
\begin{equation*}
\varepsilon \varkappa ^{+}E^{a^{+}}\left( \eta \right) <\delta C\left(
\varepsilon \varkappa ^{-}E^{a^{-}}\left( \eta \right) +m\left\vert \eta
\right\vert \right) ,~\eta \neq \emptyset
\end{equation*}%
the following estimate holds%
\begin{equation*}
\left\Vert \varepsilon B_{2}G\right\Vert _{C}\leq a\left\Vert
A_{1}(\varepsilon) G\right\Vert _{C},~G\in D_{1}
\end{equation*}%
with $a <\delta $.
\end{proposition}
\begin{remark}
Proposition~\ref{pr002} enables us to take $D(B_{2})=D_{1}$. As a
result, Remark~\ref{rem001} shows that the domain of the operator
$A_{2}(\varepsilon)$ will be $D_{0}\cap D_{1}=D_{1}$.
\end{remark}
We are now in a position to show that the operator $(
\hat{L}_{\varepsilon ,\mathrm{ren}}, \,D_{1} )$ generates semigroup
on $\mathcal{L}_{C}$. To this end we use the classical result about
the perturbation of holomorphic semigroups (see, e.g.
\cite{Kat1976}). For the convenience of the reader we formulate
below the main statement without proof:

 {\it For any $T\in\mathcal{H}(\omega), \;\omega\in(0;\,\frac{\pi}{2})$ and for any
$\epsilon > 0$ there exist positive constants $\alpha$, $\delta$
such that if the operator $A$ satisfies
\begin{equation*}
||Au||\leq a||Tu||+b||u||, \quad u\in D(T)\subset D(A),
\end{equation*}
with $a<\delta$, $b<\delta$, then
$T+A$ is a generator of a holomorphic semigroup.
In particular, if $\;b=0$, then $T+A\in
\mathcal{H}(\omega -\epsilon)$.
}

\begin{theorem}\label{theor11}
Let the functions $a^{-},a^{+}$ and the constants $m,\, \varkappa
^{-},\varkappa ^{+}, C>0$ satisfy
\begin{align}
m&>4\left( \varkappa ^{-}C+\varkappa ^{+}\right), \label{bigmort}\\
C\varkappa ^{-}a^{-}\left( x\right) &\geq 4\varkappa ^{+}a^{+}\left(
x\right) ,~x\in \mathbb{R}^{d}.\label{bigcomp}
\end{align}
Then, for any $\varepsilon > 0$
the operator $( \hat{L}_{\varepsilon ,\mathrm{ren}}, \,D_{1} ) $ is
a generator of a holomorphic semigroup $\hat{U}_{t,\varepsilon },\,
t\geq 0$ on $\mathcal{L}_{C}$.
\end{theorem}

\begin{proof}
Let  $\varepsilon > 0$ be arbitrary and fixed. By definition,
$$\hat{L}_{\varepsilon
,\mathrm{ren}}=A_{1}(\varepsilon)+A_{2}(\varepsilon).$$ The direct
application of the theorem about perturbation of holomorphic semigroups
(see the formulation above the assertion of Theorem~\ref{theor11})
to $T= A_{1}(\varepsilon)$ and
$A=A_{2}(\varepsilon)$ gives now the desired claim. It is important
to note that Proposition~\ref{pr1} enables us to
consider $\delta$ equal to $\frac{1}{2}$ in the formulation of the classical theorem introduced above.
The appearance of the multiplicand $4$ on the left-hand side of the
both assumptions in assertion of Theorem~\ref{theor11} is motivated
exactly by the latter fact.
\end{proof}

\begin{theorem}\label{co1}
Assume that the constants $\,m, \varkappa ^{-}, \varkappa ^{+}, C>0$
satisfy $$m>2\left( \varkappa ^{-}C+\varkappa ^{+}\right).$$ Then, the
operator $\hat{V}=A_{1}+A_{2}$ with the domain $D_{0}$ is
a generator of a holomorphic semigroup $\hat{U}_{t}^{V}, \,t\geq 0$ on $\mathcal{L}_{C}$.
\end{theorem}

\begin{proof}
We use the same classical result as for Theorem~\ref{theor11} in the
case: $A_{1}$ is a generator of holomorphic semigroup, $A_{2}$ is
relatively bounded w.r.t. $A_{1}$ with the boundary less then
$\frac{1}{2}$.
\end{proof}

Now we may repeat the same considerations
as at the end of Section~\ref{sect:base}.
Namely, transferring the general results
about adjoint semigroups (see, e.g., \cite{EN2000}) onto semigroup
$(\hat{U}_{t}^{V})^\ast$ in $\K_C$ we deduce that it will be weak*-continuous and
weak*-differentiable at $0$. Moreover, $\hat{V}^\ast$ will be the
weak*-generator of $\hT^\ast(t)$. This means,
in particular, that for any $G\in D(\hat{V})\subset \L_C$, $k\in D(\hat{V}^*)\subset \K_C$
\begin{equation}\label{abstrCauchy}
\frac{d}{dt}\llu G, (\hat{U}_{t}^{V})^\ast
k \rru = \llu G, \hat{V}^\ast(\hat{U}_{t}^{V})^\ast
k\rru.
\end{equation}
The explicit form of $\hat{V}^\ast$ follows
from \eqref{dual-descent}, namely, for any
$k \in D(\hat{V}^*)$
\begin{align}
\hat{V}^\ast k(\eta)=-m|\eta|k(\eta)&-\varkappa^{-}\int_{\R^{d}}\sum_{x\in\eta}a^{-}(x-y)k(\eta\cup y)dy\nonumber\\
&+ \varkappa^{+}\sum_{x\in\eta}\int_{\R^{d}}a^{+}(x-y)k(\eta\setminus x\cup y)dy.\label{vadjoint}
\end{align}
As a result, we have that for any $k_0\in
D(\hat{V}^*)$ the function $k_t=(\hat{U}_{t}^{V})^\ast k_0$ provides
a weak* solution of the following Cauchy problem
\begin{equation}\label{CauchyVlasov}
\begin{cases}
\dfrac{\partial}{\partial t} k_t = \hat{V}^\ast k_t\\[2mm]
k_t\bigr|_{t=0}=k_0.
\end{cases}
\end{equation}

In the next theorem we show that the limiting Vlasov dynamics has chaos preservation property, i.e. preserves the Lebesgue--Poisson
exponents.

\begin{theorem}\label{Vlasovscheme}
Let conditions of Theorem \ref{theor11} be satisfied and, additionally, $C\geq\frac{4}{16 e-1}$. Let $\rho_0\geq 0$ be a measurable
nonnegative function on $\X$ such that $\esssup_{x\in\X} \rho_0(x) \leq C$. Then the Cauchy problem
\eqref{CauchyVlasov} with $k_0=e_\la(\rho_0)$
has a weak* solution $k_{t}=e_\la(\rho_t)\in\K_C$, where $\rho_t$ is a unique nonnegative solution to the
Cauchy problem
\begin{equation}\label{CauchyVlasoveqn}
\begin{cases}
\dfrac{\partial}{\partial t} \rho_t(x) = \varkappa^{+}(a^{+}\ast\rho_t)(x)- \varkappa^{-}\rho_{t}(x)(a^{-}\ast \rho_{t})(x)-
m\rho_{t}(x),\\[2mm]
\rho_t \bigr|_{t=0}(x)=\rho_0(x),
\end{cases}
\end{equation}
and $\esssup_{x\in\X} \rho_{t}(x) \leq C$, $t\geq0$.
\end{theorem}
\begin{proof}

 First of all, if
\eqref{CauchyVlasoveqn} has a solution $\rho_t(x)\geq0$ then
\[
\frac{\partial}{\partial t} \rho_t(x) \leq\varkappa^{+}(a^{+}\ast\rho_t)(x) -m\rho_t(x)\]
and, therefore, $\rho_t(x)\leq r_t(x)$ where $r_t(x)$ is a solution of
the Cauchy problem
\begin{equation*}\label{CauchyEst}
\begin{cases}
\dfrac{\partial}{\partial t} r_t(x) = \varkappa^{+}(a^{+}\ast
r_t)(x)-mr_t(x)  , \\
r_t \bigr|_{t=0}(x)=\rho_0(x)\geq 0,
\end{cases}
\end{equation*}
for a.a. $x\in\X$. Hence,
\begin{align*}
r_t(x)&=e^{-(m-\varkappa^+)t}e^{\varkappa^+ tL_{a^{+}}}\rho_0(x),\\\intertext{where} (L_{a^+}f)(x)&:=\int_\X a^+(x-y)[f(y)-f(x)]dy.
\end{align*}
Since for $f\in L^\infty(\X)$ we have
$\bigl| (L_{a^+}f)(x)|\leq2\|f\|_{L^\infty(\X)}$
then, by \eqref{bigmort},
\[
r_t(x)\leq Ce^{-(m-\varkappa^+)t}e^{2\varkappa^+ t}\leq C,
\]
that yields $0\leq\rho_t(x)\leq C$.

To prove the existence and uniqueness of the solution of
\eqref{CauchyVlasoveqn}
let us fix some $T>0$
and define the Banach space $X_T=C([0;T],L^\infty(\X))$
of all continuous functions on $[0;T]$ with
values in $L^\infty(\X)$; the norm on $X_T$ is given by
\mbox{$\|u\|_T:=\max\limits_{t\in[0;T]}\|u_t\|_{L^\infty(\X)}$}.
We denote by $X_T^+$ the cone of all
nonnegative
functions
from $X_T$.

Let $\Phi$ be a  mapping which
assign to any $v\in X_T$ the solution
$u_t$
of the linear Cauchy problem
\begin{equation}\label{CauchyLin}
\begin{cases}
\dfrac{\partial}{\partial t} u_t(x) =  \varkappa^{+}(a^{+}\ast v_t)(x)- \varkappa^{-}u_{t}(x)(a^{-}\ast v_{t})(x)-
mu_{t}(x), \\
u_t \bigr|_{t=0}(x)=\rho_0(x),
\end{cases}
\end{equation}
for a.a. $x\in\X$. Therefore,
\begin{align}\label{defPhi}
(\Phi v)_t(x)=&\exp\left\{-\int_0^t\bigl(
m+\varkappa^-(a^-\ast v_s)(x)
 \bigr)ds\right\}\rho_0(x)\\&+\int_0^t \exp\left\{-\int_s^t\bigl(
m+\varkappa^-(a^-\ast v_\tau)(x)
 \bigr)d\tau\right\}\varkappa^{+}(a^{+}\ast v_s)(x)ds.\nonumber
\end{align}
We have that $v\in X_T^+$ implies $\Phi v \geq0$ as well as the estimate
\[
(\Phi v)_t(x)\leq \rho_0(x)+\varkappa^+ \|v\|_T\int_0^t e^{-(t-s)m}ds\leq C +\frac{\varkappa^+}{m} \|v\|_T,
\]
where we use the trivial inequality
\begin{equation}\label{H}
\|f\ast g\|_{L^\infty(\X)}\leq\|f\|_{L^1(\X)}\|g\|_{L^\infty(\X)},
\qquad f\in L^1(\X), \  g\in L^\infty (\X).
\end{equation}
Therefore, $\Phi v\in
X_T^+$. For simplicity of notations we denote
for  $v\in X_T^+$ \[
(Bv)(t,x)=m+\varkappa^-(a^-\ast v_t)(x)\geq m>0.
\]
Then, for any $v, w\in X_T^+$
\begin{align*}
&\bigl| (\Phi v)_t(x)-(\Phi w)_t(x) \bigr| \\
\leq & \left|\exp\left\{-\int_0^t (Bv)(s,x)ds\right\}
-\exp\left\{-\int_0^t (Bw)(s,x)ds\right\}\right|\rho_0(x)\\
&+\int_0^t \left|\exp\left\{-\int_s^t(Bv)(\tau,x)d\tau\right\}\varkappa^{+}(a^{+}\ast v_s)(x)\right. \\ &\qquad  \left.-\exp\left\{-\int_s^t(Bw)(\tau,x)d\tau\right\}\varkappa^{+}(a^{+}\ast w_s)(x)\right|ds.
\end{align*}
We have
\begin{align*}
&\left|\exp\left\{-\int_0^t (Bv)(s,x)ds\right\}
-\exp\left\{-\int_0^t (Bw)(s,x)ds\right\}\right|\\\leq&e^{-mt}\left|\exp\left\{-\int_0^t \varkappa^-(a^-\ast v_s)(x)ds\right\}
-\exp\left\{-\int_0^t \varkappa^-(a^-\ast w_s)(x)ds\right\}\right|\\\leq&
e^{-mt}\left|\int_0^t \varkappa^-(a^-\ast v_s)(x)ds-\int_0^t \varkappa^-(a^-\ast w_s)(x)ds\right|\\\leq&e^{-mt}
\varkappa^-\|v-w\|_T\cdot t\leq \frac{\varkappa^-}{em}\|v-w\|_T,
\end{align*}
where we used \eqref{H} and obvious inequalities $|e^{-a}-e^{-b}|\leq |a-b|$ for $a,b\geq0$;  $e^{-x}x\leq
e^{-1}$ for $x\geq0$.

Next, using another simple estimates for
any $a,b,p,q\geq0$
\[
|pe^{-a}-qe^{-b}|\leq e^{-a}|p-q|+qe^{-b}|e^{-(a-b)}-1|\leq
e^{-a}|p-q|+qe^{-b}|a-b|,
\]
we obtain
\begin{align*}
&\int_0^t \left|\exp\left\{-\int_s^t(Bv)(\tau,x)d\tau\right\}\varkappa^{+}(a^{+}\ast v_s)(x)\right. \\ &\qquad  \left.-\exp\left\{-\int_s^t(Bw)(\tau,x)d\tau\right\}\varkappa^{+}(a^{+}\ast w_s)(x)\right|ds\\
\leq&
\varkappa^+ \int_0^t \exp\left\{-\int_s^t(Bv)(\tau,x)d\tau\right\}\bigl|
a^+*(v_s-w_s)\bigr|(x) ds\\&
+\int_0^t \exp\left\{-\int_s^t(Bw)(\tau,x)d\tau\right\}(\varkappa^{+}a^{+}\ast w_s)(x)\\ &\quad\times\left| \int_s^t(Bv)(\tau,x)d\tau-\int_s^t(Bw)(\tau,x)d\tau\right|ds\\\leq&
\varkappa^+ \|v-w\|_T \int_0^t e^{-m(t-s)}ds\\&
+\int_0^t \exp\left\{-\int_s^t\varkappa^-(a^-\ast w_\tau)(x)d\tau\right\}(\varkappa^{+}a^{+}\ast w_s)(x)\allowdisplaybreaks[0]\\ &\quad\times e^{-m(t-s)} \int_s^t \varkappa^-(a^-\ast |v_\tau-w_{\tau}|)(x)d\tau ds
\\\intertext{and, using \eqref{bigcomp} and
the inequalities above, one can continue}\leq&
\frac{\varkappa^+}{m}\|v-w\|_T+\frac{C}{4}\frac{\varkappa^-}{em}\|v-w\|_T\\&\quad\times\int_0^t \exp\left\{-\int_s^t\varkappa^-(a^-\ast w_\tau)(x)d\tau\right\}\varkappa^-(a^{-}\ast w_s)(x)ds\\=&
\frac{\varkappa^+}{m}\|v-w\|_T+\frac{C}{4}\frac{\varkappa^-}{em}\|v-w\|_T\allowdisplaybreaks[0]\\&\quad\times\int_0^t \frac{\partial}{\partial s} \exp\left\{-\int_s^t\varkappa^-(a^-\ast w_\tau)(x)d\tau\right\}ds\\\leq&\left( \frac{\varkappa^+}{m}+\frac{C}{4}\frac{\varkappa^-}{em}\right)\|v-w\|_T.
\end{align*}

Therefore, for $v,w\in X_T^+$
\[
\|\Phi v-\Phi w\|_T\leq
\left( \frac{\varkappa^+}{m}+\Bigl(1+\frac{C}{4}\Bigr)\frac{\varkappa^-}{em}\right)\|v-w\|_T\leq \frac{4(\varkappa^++C\varkappa^-)}{m}\|v-w\|_T,
\]
if, e.g., $1+\frac{C}{4}\leq 4Ce$, that means $C\geq\frac{4}{16
e-1}$.

As a result, by
\eqref{bigmort}, $\Phi$
is a contraction mapping on the cone $X_T^+$. Taking, as usual,
$v^{(n)}=\Phi^nv^{(0)}$, $n\geq1$ for $v^{(0)}\in X_T^+$ we obtain
that $\{v^{(n)}\}\subset X_T^+$ is a fundamental sequence in $X_T$
which has, therefore, a unique limit point $v\in X_T$. Since $X_T^+$
is a closed cone we have that $v\in X_T^+$. Then, identically to the
classical Banach fixed point theorem, $v$ will be a fixed point of
$\Phi$ on $X_T$ and a unique fixed point on $X_T^+$. Then, this $v$
is the nonnegative solution of \eqref{CauchyVlasoveqn} on the
interval $[0;T]$. By the note above, $v_t(x)\leq C$. Changing
initial value in \eqref{CauchyVlasoveqn} onto $\rho_t
\bigr|_{t=T}(x)=v_T(x)$ we may extend all our considerations on the
time-interval $[T;2T]$ with the same estimate $v_t(x)\leq C$; and
so on. As a a result, \eqref{CauchyVlasoveqn} has a unique global bounded
non-negative solution $\rho_t(x)$ on $\R_+$.

Consider now
$$k_t(\eta)=e_\la(\rho_t,\eta)\in \mathcal{K}_{C},$$
then
$$
\frac{\partial}{\partial t}e_\la(\rho_t,\eta)=\sum_{x\in{\eta}}\frac{\partial\rho_{t}}{\partial t}(x)e_\la(\rho_t,\eta\setminus x).
$$
Using (\ref{CauchyVlasoveqn}) and \eqref{vadjoint}, we immediately
conclude that $k_t(\eta)=e_\la(\rho_t,\eta)$ is a solution to
\eqref{CauchyVlasov}.
\end{proof}

\noindent The main result of the paper is formulated in the next theorem. Its proof will be given in Subsection 3.4.

\begin{theorem}\label{theor22}
Under conditions of Theorem~\ref{theor11} the semigroup
$\hat{U}_{t,\varepsilon }$ converges strongly to the semigroup
$\hat{U}_{t}^{V}$ as $\varepsilon \rightarrow 0$ uniformly on any
finite intervals of time.
\end{theorem}

\subsection{Proofs}

According to Theorem~\ref{gentheor}, the statement of
Theorem~\ref{theor22} will be proved once we verify Assumptions {\bf
(A)} for the operators $\left( A_{1}(\varepsilon), D_{1}\right)$,
$\left( A_{2}(\varepsilon), D_{1}\right)$, $\varepsilon > 0$,
defined in the previous subsection. Note, that $A_{1}(0)=A_{1}$ and
$A_{2}(0)=A_{2}$ are defined on the domain $D_{0}$.

In the following proposition we verify Assumption 2(a) of {\bf (A)}.
\begin{proposition}\label{Prop2.1}
Let $\lambda > 0$ then
\begin{equation*}
\left( A_{1}(\varepsilon)-\lambda 1\!\!1 \right) ^{-1}\overset{s}{%
\longrightarrow }\left( A_{1}-\lambda 1\!\!1 \right) ^{-1},\varepsilon \rightarrow
0.
\end{equation*}
\end{proposition}

\begin{proof}
For any $G\in \mathcal{L}_{C}$%
\begin{align*}
&\left\Vert \left( A_{1}(\varepsilon)-\lambda 1\!\!1 \right)
^{-1}G-\left(
A_{1}-\lambda 1\!\!1 \right) ^{-1}G\right\Vert _{C} \\
=&\int_{\Gamma _{0}}\left\vert G\left( \eta \right) \left( \frac{1}{%
-m\left\vert \eta \right\vert -\varepsilon \varkappa ^{-}E^{a^{-}}\left(
\eta \right) -\lambda }-\frac{1}{-m\left\vert \eta \right\vert -\lambda }%
\right) \right\vert C^{\left\vert \eta \right\vert }d\lambda \left( \eta
\right)  \\
= &\int_{\Gamma _{0}}\left\vert G\left( \eta \right) \right\vert
F_{\varepsilon }\left( \eta \right) C^{\left\vert \eta \right\vert
}d\lambda \left( \eta \right) ,
\end{align*}%
where
\begin{equation*}
F_{\varepsilon }\left( \eta \right) :=\frac{\varepsilon \varkappa
^{-}E^{a^{-}}\left( \eta \right) }{\left( m\left\vert \eta \right\vert
+\varepsilon \varkappa ^{-}E^{a^{-}}\left( \eta \right) +\lambda \right)
\left( m\left\vert \eta \right\vert +\lambda \right) },~~\eta \in \Gamma
_{0}.
\end{equation*}%
Since $0\leq F_{\varepsilon }\left( \eta \right) <1/\lambda$ and $F_{\varepsilon
}\left( \eta \right) \rightarrow 0$ as $\varepsilon \rightarrow 0$ for any $\eta
\in \Gamma _{0}$, we get the desired statement.
\end{proof}
Next we check Assumption 2(b) of  {\bf (A)}.
\begin{proposition}\label{Prop2.2}
Let $\lambda > 0$ be arbitrary and fixed. Then
\begin{equation*}
\sup_{\varepsilon \geq 0}\left\Vert \left( A_{1}(\varepsilon)-\lambda 1\!\!1 \right)
^{-1}\right\Vert \leq \left\Vert \left( A_{1}-\lambda 1\!\!1 \right) ^{-1}\right\Vert .
\end{equation*}
\end{proposition}

\begin{proof}
For any $G\in \mathcal{L}_{C}$ and any $\varepsilon >0$
\begin{align*}
&\left\Vert \left( A_{1}(\varepsilon)-\lambda 1\!\!1 \right)
^{-1}G\right\Vert _{C} \\
=&\int_{\Gamma _{0}}\left\vert G\left( \eta \right) \right\vert \frac{1}{%
m\left\vert \eta \right\vert +\varepsilon \varkappa ^{-}E^{a^{-}}\left( \eta
\right) +\lambda }C^{\left\vert \eta \right\vert }d\lambda \left( \eta
\right)  \\
\leq &\int_{\Gamma _{0}}\left\vert G\left( \eta \right) \right\vert \frac{1%
}{m\left\vert \eta \right\vert +\lambda }C^{\left\vert \eta \right\vert
}d\lambda \left( \eta \right) =\left\Vert \left( A_{1}-\lambda 1\!\!1 \right)
^{-1}G\right\Vert _{C} \\
\leq &\left\Vert \left( A_{1}-\lambda 1\!\!1 \right) ^{-1}\right\Vert \cdot
\left\Vert G\right\Vert _{C}.
\end{align*}
This finishes the proof.
\end{proof}
Assumption 2(c) of  {\bf (A)} is proved in the next Proposition.
\begin{proposition}\label{2(cc)}
Let conditions of Theorem \ref{theor11} be satisfied.
Then, for any $\lambda>0$
\begin{equation}
\sup_{\varepsilon\geq 0}\left\Vert A_{2}(\varepsilon)\left( A_{1}(\varepsilon)-\lambda 1\!\!1 \right)
^{-1}\right\Vert < \frac{1}{2}\label{2(c)}
\end{equation}

\end{proposition}
\begin{proof} First we prove assertion for $\varepsilon=0$.
Since $D\left( A_{1}\right)= D\left( A_{2}\right)=D_{0} $ and $Ran\left(
\left( A_{1}-\lambda 1\!\!1 \right) ^{-1}\right) =D\left( A_{1}\right)$, the operator $A_{2}\left( A_{1}-\lambda 1\!\!1 \right) ^{-1}$ is well defined.
Next, inequality \eqref{bigmort} and Proposition~\ref{pr001} yields%
\begin{equation}
\left\Vert A_{2}\left( A_{1}-\lambda 1\!\!1 \right) ^{-1}\right\Vert <\frac{1}{4}.
\label{star}
\end{equation}%
Indeed,
\begin{equation*}
\left\Vert A_{2}G\right\Vert _{C}\leq a\left\Vert A_{1}G\right\Vert
_{C}<a\left\Vert \left( A_{1}-\lambda 1\!\!1 \right) G\right\Vert _{C}
\end{equation*}
with $a<\frac{1}{4}$. Therefore,
\begin{equation*}
\left\Vert A_{2}\left( A_{1}-\lambda 1\!\!1 \right) ^{-1}G\right\Vert _{C}<\frac{1}{%
4}\left\Vert G\right\Vert _{C},
\end{equation*}%
and \eqref{star} is proved.

Now, let $\varepsilon > 0$ be arbitrary and fixed.
The main arguments we use to show
$$
\left\Vert A_{2}(\varepsilon)\left( A_{1}(\varepsilon)-\lambda 1\!\!1 \right)
^{-1}\right\Vert < \frac{1}{2}
$$
are the following:

1) $D\left( A_{1}(\varepsilon) \right) =D_{1} \subset D_{0}=D\left(
A_{2}\right) $. Hence, $A_{2}\left( A_{1}(\varepsilon)-\lambda 1\!\!1
\right) ^{-1}$ is well-defined on $\mathcal{L}_{C}$. Moreover,
Proposition~\ref{pr001} implies

\begin{equation*}
\left\Vert A_{2}\left( A_{1}(\varepsilon)-\lambda 1\!\!1 \right)
^{-1}\right\Vert <\frac{1}{4}, \quad \varepsilon >0.
\end{equation*}

2) $D\left( B_{2}\right) =D\left( A_{1}(\varepsilon) \right) =D_{1} $ and for any
$\varepsilon >0$
\begin{equation*}
\left\Vert \varepsilon B_{2}\left( A_{1}(\varepsilon)-\lambda 1\!\!1
\right) ^{-1}\right\Vert <\frac{1}{4},
\end{equation*}
which follows from Proposition~\ref{pr002}.

3) Since $A_{2}(\varepsilon):=A_{2}+\varepsilon B_{2}$, we have
\begin{equation}
\left\Vert A_{2}(\varepsilon)\left( A_{1}(\varepsilon)-\lambda 1\!\!1 \right)
^{-1}\right\Vert <\frac{1}{2}.  \label{againstar}
\end{equation}
The latter concludes the proof.
\end{proof}

We set%
\begin{align*}
Q_{\varepsilon } =\left( A_{2}(\varepsilon)\left( A_{1}(\varepsilon)-\lambda 1\!\!1 \right) ^{-1}+1\right) ^{-1},
\quad\quad
Q =\left( A_{2}\left( A_{1}-\lambda 1\!\!1 \right) ^{-1}+1\!\!1\right) ^{-1}.
\end{align*}%
%
In order to verify Assumption 2(d) of {\bf (A)} we have to show that $Q_{\varepsilon }\overset{s}{\longrightarrow }Q$ as $\varepsilon\rightarrow 0$.

Suppose that we can show that
\begin{equation} \label{3star}
\begin{aligned}
A_{2}\left( A_{1}(\varepsilon)-\lambda 1\!\!1 \right) ^{-1}&\overset{s}{%
\longrightarrow }A_{2}\left( A_{1}-\lambda 1\!\!1 \right)
^{-1},&\varepsilon
\rightarrow 0.  \\
\varepsilon B_{2}\left( A_{1}(\varepsilon)-\lambda 1\!\!1 \right) ^{-1}&\overset%
{s}{\longrightarrow }0,&\varepsilon  \rightarrow 0.
\end{aligned}%
\end{equation}
Then,%
\begin{align*}
C_{\varepsilon } :=&A_{2}(\varepsilon)\left(
A_{1}(\varepsilon)-\lambda 1\!\!1 \right) ^{-1} \\
=&A_{2}\left( A_{1}+\varepsilon B_{1}-\lambda 1\!\!1 \right)
^{-1}+\varepsilon
B_{2}\left( A_{1}+\varepsilon B_{1}-\lambda 1\!\!1 \right) ^{-1}\overset{s}{%
\longrightarrow }A_{2}\left( A_{1}-\lambda 1\!\!1 \right) ^{-1}
\end{align*}%
To check
\begin{equation}
Q_{\varepsilon }=\left( C_{\varepsilon }+1\!\!1\right) ^{-1}\overset{s}{%
\longrightarrow }Q  \label{twostar}
\end{equation}%
we proceed as follows:
\begin{align*}
&\left( C_{\varepsilon }+1\!\!1\right) ^{-1}-Q \\
=&\left( C_{\varepsilon }+1\!\!1\right) ^{-1}-\left( A_{2}\left(
A_{1}-\lambda 1\!\!1
\right) ^{-1}+1\!\!1\right) ^{-1} \\
=&\left( C_{\varepsilon }+1\!\!1\right) ^{-1}\left( A_{2}\left(
A_{1}-\lambda 1\!\!1 \right) ^{-1}+1\!\!1-C_{\varepsilon }-1\!\!1\right) \left(
A_{2}\left( A_{1}-\lambda 1\!\!1
\right) ^{-1}+1\!\!1\right) ^{-1} \\
=&\left( C_{\varepsilon }+1\!\!1\right) ^{-1}\left( A_{2}\left(
A_{1}-\lambda 1\!\!1 \right) ^{-1}-C_{\varepsilon }\right) \left(
A_{2}\left( A_{1}-\lambda 1\!\!1 \right) ^{-1}+1\!\!1\right) ^{-1}.
\end{align*}%
Assuming \eqref{3star} it is obvious now that convergence \eqref%
{twostar} is equivalent to
\begin{equation*}
\sup_{\varepsilon >0}\left\Vert \left( C_{\varepsilon }+1\!\!1\right)
^{-1}\right\Vert <\infty,
\end{equation*}
which is clear from
\begin{equation*}
\left\Vert \left( C_{\varepsilon }+1\!\!1\right) ^{-1}\right\Vert \leq \frac{1}{%
1-\left\Vert C_{\varepsilon }\right\Vert }\quad \mathrm{and}\quad\left\Vert C_{\varepsilon }\right\Vert <\frac{1}{2}.
\end{equation*}
The last bound we conclude from \eqref{againstar}.
As a result we shall have established Theorem~\ref{theor22} if we show
\eqref{3star}.

\begin{lemma}\label{L1}
$A_{2}\left( A_{1}(\varepsilon)-\lambda 1\!\!1 \right) ^{-1}\overset{s}{%
\longrightarrow }A_{2}\left( A_{1}-\lambda 1\!\!1 \right) ^{-1},\; \mathrm{as}\; \varepsilon
\rightarrow 0.$
\end{lemma}

\begin{proof}[Proof of Lemma~\ref{L1}]
Proposition~\ref{pr001} and
$$D\left(A_{1}(\varepsilon)\right) =D_{1}\subset D\left( A_{1}\right)=D\left( A_{2}\right) =D_{0}$$ leads to the following formula
\begin{equation*}
A_{2}\left( A_{1}(\varepsilon)-\lambda 1\!\!1 \right) ^{-1}=A_{2}\left(
A_{1}-\lambda 1\!\!1 \right) ^{-1}\left( A_{1}-\lambda 1\!\!1 \right) \left(
A_{1}(\varepsilon)-\lambda 1\!\!1 \right) ^{-1}.
\end{equation*}%
Now, we are left with the task to show that
\begin{equation*}
\left( A_{1}-\lambda 1\!\!1 \right) \left( A_{1}(\varepsilon)-\lambda 1\!\!1 \right)
^{-1}\overset{s}{\longrightarrow }1, \quad \mathrm{as}\quad \varepsilon
\rightarrow 0.
\end{equation*}%
But, for any $G\in \mathcal{L}_{C}$%
\begin{align*}
&\left\Vert \left( \left( A_{1}-\lambda 1\!\!1 \right) \left(
A_{1}(\varepsilon)-\lambda 1\!\!1 \right) ^{-1}-1\!\!1\right) G\right\Vert _{C} \\
=&\int_{\Gamma _{0}}\left\vert \frac{m\left\vert \eta \right\vert
+\lambda }{m\left\vert \eta \right\vert +\varepsilon \varkappa
^{-}E^{a^{-}}\left( \eta \right) +\lambda }-1\right\vert \left\vert
G\left( \eta \right)
\right\vert C^{\left\vert \eta \right\vert }d\lambda \left( \eta \right)  \\
=&\int_{\Gamma _{0}}\frac{\varepsilon \varkappa ^{-}E^{a^{-}}\left(
\eta \right) }{m\left\vert \eta \right\vert +\varepsilon \varkappa
^{-}E^{a^{-}}\left( \eta \right) +\lambda }\left\vert G\left( \eta
\right) \right\vert C^{\left\vert \eta \right\vert }d\lambda \left(
\eta \right) \rightarrow 0,\quad \mathrm{as}\quad \varepsilon
\rightarrow 0
\end{align*}
due to the Lebesgue's dominated convergence theorem.

\begin{lemma}\label{L2}
$\varepsilon B_{2}\left( A_{1}(\varepsilon)-\lambda 1\!\!1 \right) ^{-1}%
\overset{s}{\longrightarrow }0,\; \mathrm{as}\; \varepsilon
\rightarrow 0.$
\end{lemma}

\begin{proof}[Proof of Lemma~\ref{L2}]
Since $\left\Vert  B_{2}G\right\Vert _{C}\leq \frac{1}{4}%
\left\Vert  B_{1}G\right\Vert _{C}$, we have to show that%
\begin{equation*}
\left\Vert \varepsilon B_{1}\left( A_{1}(\varepsilon)-\lambda 1\!\!1 \right)
^{-1}G\right\Vert _{C}\rightarrow 0,\quad \mathrm{as}\quad\varepsilon \rightarrow 0.
\end{equation*}%
But,
\begin{align*}
&\left\Vert \varepsilon B_{1}\left( A_{1}(\varepsilon)-\lambda 1\!\!1
\right) ^{-1}G\right\Vert _{C}\\&=\int_{\Gamma
_{0}}\frac{\varepsilon \varkappa ^{-}E^{a^{-}}\left( \eta \right)
}{m\left\vert \eta \right\vert +\varepsilon \varkappa
^{-}E^{a^{-}}\left( \eta \right) +\lambda }\left\vert G\left( \eta
\right) \right\vert C^{\left\vert \eta \right\vert }d\lambda \left(
\eta \right) \rightarrow 0,\quad\varepsilon\rightarrow 0. \qedhere
\end{align*}
\end{proof}
\noindent The last two lemmas conclude the proof of the main Theorem.
\end{proof}

\begin{remark}
Under assumptions of Proposition \ref{2(cc)} we get the following representation
for the resolvents of $\hat{V}$ and $\hat{L}_{\varepsilon ,\mathrm{ren}}$
\begin{equation*}
\left( \hat{V}-\lambda 1\!\!1 \right)
^{-1}=\left( A_{1}+A_{2}-\lambda 1\!\!1 \right) ^{-1}=\left( A_{1}-\lambda 1\!\!1 \right)
^{-1}\left( A_{2}\left( A_{1}-\lambda 1\!\!1 \right) ^{-1}+1\!\!1\right) ^{-1},
\end{equation*}%
\begin{align}
\left( \hat{L}_{\varepsilon ,\mathrm{ren}}-\lambda 1\!\!1 \right) ^{-1}=&\left( A_{1}(\varepsilon)+A_{2}(\varepsilon)-\lambda 1\!\!1
\right)
^{-1}  \label{longexpansresolv2} \\
=&\left( A_{1}(\varepsilon)-\lambda 1\!\!1 \right) ^{-1}\left(
A_{2}(\varepsilon)\left( A_{1}(\varepsilon)-\lambda 1\!\!1 \right)
^{-1}+1\!\!1\right) ^{-1},\quad\quad \lambda >0.  \nonumber
\end{align}%
\end{remark}

\subsection*{Acknowledgements}
The financial support of DFG through the SFB 701 (Bielefeld
University) and German-Ukrainian Projects 436 UKR 113/97 is
gratefully acknowledged. This work was partially supported by the
Marie Curie ``Transfer of Knowledge'' programme, project TODEQ
MTKD-CT-2005-030042 (Warsaw, IMPAN). O.K. is very thankful to Prof.
J. Zemanek for fruitful and stimulating discussions.

\end{document}